\documentclass[a4paper,reqno,portrait,tbtags,11pt]{amsart}
\usepackage[pdftex]{graphicx}
\usepackage{amssymb}
\usepackage{amsmath}
\usepackage{amsfonts}
\usepackage{xcolor} 
\usepackage{geometry}

\newtheorem{theorem}{Theorem}[section]
\theoremstyle{plain}

\newtheorem{special case}{Special Case}

\newtheorem{corollary}{Corollary}[section]

\newtheorem{example}{Example}[section]

\newtheorem{lemma}{Lemma}[section]

\newtheorem{remark}{Remark}[section]

\numberwithin{equation}{section}

\begin{document}
\title[Nonlinear sampling Kantorovich operators]{Quantitative estimates for
nonlinear\\ sampling Kantorovich operators}
\author{Nursel \c{C}etin}
\author{Danilo Costarelli}
\author{Gianluca Vinti}
\address{Ankara Hac\i\ Bayram Veli University, Polatl\i\ Faculty of Science
and Letters, Department of Mathematics, 06900, Ankara, Turkey}
\email{nurselcetin07@gmail.com}
\address{University of Perugia, Department of Mathematics and Computer
Science 1, Via Vanvitelli, 06123 Perugia, Italy}
\email{danilo.costarelli@unipg.it}
\address{University of Perugia, Department of Mathematics and Computer
Science 1, Via Vanvitelli, 06123 Perugia, Italy}
\email{gianluca.vinti@unipg.it}
\subjclass{41A25, 41A35, 46E30, 47A58, 47B38, 94A12.}
\keywords{Nonlinear sampling Kantorovich operators, Orlicz spaces, Modulus
of smoothness, Quantitative estimates, Lipschitz classes.}

\begin{abstract}
In this paper, we establish quantitative estimates for nonlinear
sampling Kantorovich operators in terms of the modulus of continuity in the
setting of Orlicz spaces. 
This general frame
allows us to
directly deduce 
some quantitative estimates of
approximation in $L^{p}$-spaces, $1\leq p<\infty $,
and in other well-known instances of Orlicz spaces, such as the Zygmung and the exponential spaces. Further, the qualitative order of approximation has been obtained assuming $f$ in suitable Lipschitz classes. The above estimates achieved in the general setting of Orlicz spaces, have been also improved in the $L^p$-case, using a direct approach suitable to this context. 
At the end, we consider the particular cases of the nonlinear sampling Kantorovich operators constructed by using some special kernels.
\end{abstract}

\maketitle
\author{}

\section{Introduction}

The theory of nonlinear integral operators in connection with approximation
problems has been started by Musielak (see e.g. \cite{M1}-\cite{M4}) and subsequently, it has been
extensively developed by Bardaro, Musielak and Vinti in \cite{Bardaro1}, and
also studied in various papers by other authors (see e.g. \cite{BV98, VI2001, M, L1, L2, AV3, BKV1, AV4, C-V2,COSA,CSAMBV1}).

The linear version of the sampling Kantorovich type operators has been first introduced in \cite%
{Bardaro2} in one-dimensional setting and there, some approximation results
in the setting of Orlicz spaces $L^\varphi(\mathbb{R})$ have been achieved.

  It is well-known that, Orlicz spaces are very general spaces including,
among its various special cases, the $L^{p}$-spaces (see e.g. \cite%
{Musielak2, Musielak, Bardaro1}). Subsequently, these operators were extended in \cite{Vinti,VIZA1} to the nonlinear case. The order of approximation for nonlinear
sampling Kantorovich operators have been studied in \cite{C-V2} considering functions in suitable Lipschitz classes both in the space of
uniformly continuous and bounded functions and in Orlicz spaces. Results concerning the multidimensional version of these operators have been
obtained in \cite{C-V3}.

In the last forty years, the study of approximation results by sampling-type
operators (in linear and nonlinear cases) has been a wide research
area both red from a theoretical and an application point of view, such as signal and
image processing. In particular, sampling type operators (in their multivariate version) can be used in order to reconstruct and approximate
images, see e.g. \cite{ R, K, CSV2, C-V7}.

Concerning the problem of the order of approximation
for the (linear) sampling Kantorovich operators, a quantitative estimate in the setting of Orlicz spaces in terms of modulus of continuity
has been very recently established in \cite{C-V4}. On the other hand,
quantitative estimates with respect to the Jordan variation for
sampling-type operators have been obtained in \cite{L3} exploiting a
suitable modulus of smoothness for the space of absolutely continuous
functions $AC\left( 
\mathbb{R}
\right)$. 
However, a quantitative approach for nonlinear sampling
Kantorovich operators has not been addressed as yet.

In the present paper, we prove some quantitative estimates for the nonlinear sampling
Kantorovich operators using the modulus of continuity of $L^\varphi(\mathbb{R})$.
Further, the qualitative order of approximation is established for functions belonging to suitable Lipschitz classes. In the particular case of $L^p$-approximation, we directly established a quantitative estimate for the order of approximation, with the main purpose to obtain a sharper estimate than that one achieved in the general case.
If the latter estimated is applied for the linear version of the sampling Kantorovich operators, we become able to improve the result that could be derived from Theorem 3.1 of \cite{C-V4}.
Finally, we give some concrete examples of nonlinear sampling Kantorovich operators constructed by using  Fej\'{e}r and B-spline kernels, establishing some particular results in these instances.

\section{Preliminaries} \label{sec2}

In this section, we recall the necessary background material related to
Orlicz spaces used throughout the paper.

We denote by $C\left( 
\mathbb{R}
\right) $ the space of all uniformly continuous and bounded functions $f:%
\mathbb{R}
\rightarrow 
\mathbb{R}
$ endowed with the norm $\left \Vert \cdot \right \Vert _{\infty }.$ Also, $%
C_{c}\left( 
\mathbb{R}
\right) $ is the subspace of $C\left( 
\mathbb{R}
\right) $ consisting of functions with compact support and $M\left( 
\mathbb{R}
\right) $ is the linear space of Lebesgue measurable functions $f:%
\mathbb{R}
\rightarrow 
\mathbb{R}
$ (or $%
\mathbb{C}
$).

A function $\varphi :%
\mathbb{R}
_{0}^{+}\rightarrow 
\mathbb{R}
_{0}^{+}$ is said to be a $\varphi -$function if it satisfies the following
conditions:

$\left( \Phi 1\right) $ $\varphi $ is a continuous function with $\varphi
\left( 0\right) =0;$

$\left( \Phi 2\right) $ $\varphi $ is a non-decreasing function and $\varphi
\left( u\right) >0$ for every $u>0;$

$\left( \Phi 3\right) \lim \limits_{u\rightarrow +\infty }\varphi \left(
u\right) =+\infty .$

Let us introduce the functional $I^{\varphi }$ associated to the $\varphi -$%
function $\varphi $ and defined by 
\begin{equation*}
I^{\varphi }\left[ f\right] :=\int\limits_{%
\mathbb{R}
}\varphi \left( \left\vert f\left( x\right) \right\vert \right) dx,
\end{equation*}%
for every $f\in M\left( 
\mathbb{R}
\right) .$ It is well-known that $I^{\varphi }$ is a modular functional
(see, e.g., \cite{Musielak, Bardaro1}) and the Orlicz space generated by $%
\varphi $ is defined by%
\begin{equation*}
L^{\varphi }\left( 
\mathbb{R}
\right) :=\left\{ f\in M\left( 
\mathbb{R}
\right) :I^{\varphi }\left[ \lambda f\right] <+\infty ,\text{ for some }%
\lambda >0\right\} .
\end{equation*}%
$L^{\varphi }\left( 
\mathbb{R}
\right) $ contains the subspace $E^{\varphi }\left( 
\mathbb{R}
\right) $ of all finite elements of $L^{\varphi }\left( 
\mathbb{R}
\right) ,$ i.e.,%
\begin{equation*}
E^{\varphi }\left( 
\mathbb{R}
\right) =\left\{ f\in L^{\varphi }\left( 
\mathbb{R}
\right) :I^{\varphi }\left[ \lambda f\right] <+\infty ,\text{ for every }%
\lambda >0\right\} .
\end{equation*}%
In general $E^{\varphi }\left( 
\mathbb{R}
\right) $ is a proper subspace of $L^{\varphi }\left( 
\mathbb{R}
\right) $ and they coincide if and only if the so-called $\Delta _{2}-$%
condition on $\varphi $\ is satisfied, i.e. there exists a constant $M>0$ such that %
\begin{equation}
\text{ \ \ \ \ \ \ \ \ \ }\frac{%
\varphi \left( 2u\right) }{\varphi \left( u\right) }\leq M,\text{ for every }%
u>0.  \tag{$\Delta _{2}$}
\end{equation}%
Examples of $\varphi -$functions\ satisfying $\Delta _{2}-$condition
are $\varphi \left( u\right) =u^{p}$ for $1\leq p<\infty $ (in this case $%
L^{\varphi }\left( 
\mathbb{R}
\right) =L^{p}\left( 
\mathbb{R}
\right) $) or 
$\varphi_{\alpha, \beta} \left( u\right) =u^{\alpha }\ln ^{\beta }\left(
e+u\right) $
for $\alpha \geq 1,\beta >0$ (in this case, 
the Orlicz spaces $L^{\alpha ,\beta }\left( 
\mathbb{R}
\right)$ is the interpolation space $L^{\alpha }\log ^{\beta }L\left(\mathbb{R}\right)$).

A concept of convergence in Orlicz spaces, called \textit{modular convergence%
}, was introduced in \cite{Musielak2}.

We say that a net of functions $\left( f_{w}\right) _{w>0}\subset L^{\varphi
}\left( 
\mathbb{R}
\right) $ is modularly convergent to $f\in L^{\varphi }\left( 
\mathbb{R}
\right) ,$ if there exists $\lambda >0$ such that%
\begin{equation}
I^{\varphi }\left[ \lambda \left( f_{w}-f\right) \right] =\int \limits_{%
\mathbb{R}
}\varphi \left( \lambda \left \vert f_{w}\left( x\right) -f\left( x\right)
\right \vert \right) dx\rightarrow 0,\text{ as }w\rightarrow +\infty .
\label{I}
\end{equation}%
Now, we can recall the definition of the modulus of continuity in Orlicz
spaces $L^{\varphi }\left( 
\mathbb{R}
\right) ,$ with respect to the modular $I^{\varphi }.$ For any fixed $f\in
L^{\varphi }\left( 
\mathbb{R}
\right) $ and for a suitable $\lambda >0,$ we denote%
\begin{equation}
\omega \left( f,\delta \right) _{\varphi }:=\sup_{\left \vert t\right \vert
\leq \delta }I^{\varphi }\left[ \lambda \left( f\left( \cdot +t\right)
-f\left( t\right) \right) \right] ,  \label{w}
\end{equation}%
with $\delta >0.$

In order to recall the class of operators we work with, we need some
additional concepts.

Let $\Pi:=\left( t_{k}\right) _{k\in 
\mathbb{Z}
}$ be a sequence of real numbers such that $-\infty <t_{k}<t_{k+1}<+\infty $
for every $k\in 
\mathbb{Z}
,$ $\lim \limits_{k\rightarrow \pm \infty }t_{k}=\pm \infty $ and there are
two positive constants $\Delta ,$ $\delta $ such that $\delta \leq \Delta
_{k}:=t_{k+1}-t_{k}\leq \Delta ,$ for every $k\in 
\mathbb{Z}
.$

In what follows, a function $\chi :%
\mathbb{R}
\times 
\mathbb{R}
\rightarrow 
\mathbb{R}
$ will be called a {\em kernel} if it satisfies the following conditions:

$\left( \chi 1\right) $ $k\rightarrow \chi \left( wx-t_{k},u\right) \in \ell
^{1}\left( 
\mathbb{Z}
\right) $, for every $\left( x,u\right) \in 
\mathbb{R}
^{2}$ and $w>0;$

$\left( \chi 2\right) $ $\chi \left( x,0\right) =0,$ for every $x\in 
\mathbb{R}
;$

$\left( \chi 3\right) $ $\chi $ is a $\left( L,\psi \right) $-Lipschitz
kernel, i.e., there exist a measurable function $L:%
\mathbb{R}
\rightarrow 
\mathbb{R}
_{0}^{+}$ and a $\varphi $-function $\psi ,$ such that%
\begin{equation*}
\left \vert \chi \left( x,u\right) -\chi \left( x,v\right) \right \vert \leq
L\left( x\right) \psi \left( \left \vert u-v\right \vert \right) ,
\end{equation*}%
for every $x,$ $u,$ $v\in 
\mathbb{R}
;$

$\left( \chi 4\right) $ there exists $\theta _{0}>0$ such that%
\begin{equation*}
\mathcal{T} _{w}\left( x\right) :=\sup_{u\neq 0}\left \vert \frac{1}{u}\sum
\limits_{k\in 
\mathbb{Z}
}\chi \left( wx-t_{k},u\right) -1\right \vert =\mathcal{O}\left( w^{-\theta_{0}}\right)
\end{equation*}%
as $w\rightarrow +\infty ,$ uniformly with respect to $x\in 
\mathbb{R}
.$

Moreover, we assume that the function $L$ of condition $\left( \chi 3\right) 
$ satisfies the following properties:

$\left( L1\right) $ $L\in L^{1}\left( 
\mathbb{R}
\right) $ and is bounded in a neighborhood of the origin;

$\left( L2\right) $ there exists $\beta _{0}>0$ such that%
\begin{equation*}
m_{\beta _{0},\Pi }\left( L\right) :=\sup_{x\in 
\mathbb{R}
}\sum \limits_{k\in 
\mathbb{Z}
}L\left( wx-t_{k}\right) \left \vert wx-t_{k}\right \vert ^{\beta
_{0}}<+\infty,
\end{equation*}
i.e., the absolute moment of order $\beta_0$ is finite.

Then, nonlinear sampling Kantorovich operators for a given kernel $\chi $
are defined by%
\begin{equation*}
\left( S_{w}f\right) \left( x\right) :=\sum \limits_{k\in 
\mathbb{Z}
}\chi \left( wx-t_{k},\frac{w}{\Delta _{k}}\int%
\limits_{t_{k}/w}^{t_{k+1}/w}f\left( u\right) du\right) ,\text{ \ \ \ \ \ \
\ }x\in 
\mathbb{R}
,
\end{equation*}%
where $f:%
\mathbb{R}
\rightarrow 
\mathbb{R}
$ is a locally integrable function such that the above series is convergent
for every $x\in 
\mathbb{R}
.$

The following lemma will be needed in the proof of our main theorems.

\begin{lemma}
\label{Lem1} (see \cite{Bardaro2}) Let $L$ be a function satisfying conditions $%
\left( L1\right) $ and $\left( L2\right) .$ Then, we have%
\begin{equation*}
m_{0,\Pi }\left( L\right) :=\sup_{u\in 
\mathbb{R}
}\sum\limits_{k\in 
\mathbb{Z}
}L\left( u-t_{k}\right) <+\infty .
\end{equation*}
\end{lemma}


\section{Main Results}

In this section, firstly we give the following quantitative estimate for
nonlinear sampling Kantorovich operators by using the modulus of continuity in
Orlicz spaces. For this aim, we need a growth condition on the composition
of the function $\varphi ,$ which generates the Orlicz space and the
function $\psi $ of condition $\left( \chi 3\right) $.

Namely, given a $\varphi -$function $\varphi $, we require that there exists
a $\varphi -$function $\eta $ such that, for every $\lambda \in \left(
0,1\right) ,$ there exists a constant $C_{\lambda }\in \left( 0,1\right) $
satisfying%
\begin{equation}
\text{\ \ \ \ \ \ \ \ \ \ \ }\varphi \left( C_{\lambda }\psi \left(
u\right) \right) \leq \eta \left( \lambda u\right) ,  \tag{H}
\end{equation}%
for every $u\in 
\mathbb{R}
_{0}^{+}$,\ where $\psi $ is the $\varphi $-function of condition $\left( \chi
3\right) $ (see \cite{Bardaro1,Vinti,C-V1}).

\begin{theorem}
\label{Th1}Let $\varphi $ be a convex $\varphi -$function$.$ Suppose that $%
\varphi $ satisfies condition (H) with $\eta $ convex, $f\in L^{\varphi
+\eta }$\bigskip $\left( 
\mathbb{R}
\right) $ and also for any fixed $0<\alpha <1,$ we have%
\begin{equation}
w\int\limits_{\left\vert y\right\vert >1/w^{\alpha }}L\left( wy\right)
dy\leq M_{1}w^{-\alpha _{0}},\text{ as }w\rightarrow +\infty ,  \label{a1}
\end{equation}%
for suitable positive constants $M_{1},$ $\alpha _{0}$ depending on $\alpha $
and $L.$ Then, 
there exist $\mu >0$ and $\lambda_0>0$ such that%
\begin{eqnarray*}
I^{\varphi }\left[ \mu \left( S_{w}f-f\right) \right] &\leq &\frac{%
\left\Vert L\right\Vert _{1}}{3\delta m_{0,\Pi }\left( L\right) }\omega
\left( f,1/w^{\alpha }\right) _{\eta }+\frac{M_{1}I^{\eta }\left[ \lambda
_{0}f\right] }{3\delta m_{0,\Pi }\left( L\right) }w^{-\alpha _{0}} \\
&&+\frac{\Delta }{3\delta }\omega \left( f,\Delta /w\right) _{\eta }+\frac{%
I^{\varphi }\left[ \lambda _{0}f\right] }{3}w^{-\theta _{0}},
\end{eqnarray*}%
for every sufficiently large $w>0,$ where $m_{0,\Pi }\left( L\right)
<+\infty $ in view of Lemma \ref{Lem1} and $\theta _{0}>0$ is the constant
of condition $(\chi 4).$ Particularly, if $\mu >0$ is sufficiently small,
this inequality implies the modular convergence of nonlinear sampling
Kantorovich operators $S_{w}f$ to $f.$

\begin{proof}
Let $\lambda _{0}>0$ such that $I^{\varphi }\left[ \lambda _{0}f\right]
<+\infty $. Also, we fix $\lambda >0$ such that%
\begin{equation*}
\lambda <\min \left\{ 1,\frac{\lambda _{0}}{2}\right\} .
\end{equation*}%
In correspondence to $\lambda ,$ by condition (H) there exists $C_{\lambda }\in
\left( 0,1\right) $ such that $\varphi \left( C_{\lambda }\psi \left(
u\right) \right) \leq \eta \left( \lambda u\right) ,$ $u\in 
\mathbb{R}
_{0}^{+}.$ By condition $(\chi 4),$ there exists $M_{2}>0$ such that%
\begin{equation*}
\mathcal{T} _{w}\left( x\right) \leq M_{2}w^{-\theta _{0}},
\end{equation*}%
uniformly with respect to $x\in 
\mathbb{R}
,$ for sufficiently large $w>0.$ Let us choose $\mu >0$ such that 
\begin{equation*}
\mu \leq \min \left\{ \frac{C_{\lambda }}{3m_{0,\Pi }\left( L\right) },\frac{%
\lambda _{0}}{3M_{2}}\right\} .
\end{equation*}%
Taking into account that $\varphi $ is convex and non-decreasing, for $\mu
>0,$ we can write%
\begin{eqnarray*}
&&I^{\varphi }\left[ \mu \left( S_{w}f-f\right) \right] \\
&\leq &\frac{1}{3}\left\{ \int\limits_{%
\mathbb{R}
}\varphi \left( 3\mu \left\vert \left( S_{w}f\right) \left( x\right)
-\sum\limits_{k\in 
\mathbb{Z}
}\chi \left( wx-t_{k},\frac{w}{\Delta _{k}} \hskip-.3cm \int%
\limits_{t_{k}/w}^{t_{k+1}/w} \hskip-.3cm  f\left( u+x-t_{k}/w\right) du\right)
\right\vert \right) dx\right. \\
&&\left. +\int\limits_{%
\mathbb{R}
}\varphi \left( 3\mu \left\vert \sum\limits_{k\in 
\mathbb{Z}
}\chi \left( wx-t_{k},\frac{w}{\Delta _{k}} \hskip-.3cm  \int
\limits_{t_{k}/w}^{t_{k+1}/w} \hskip-.3cm  f\left( u+x-t_{k}/w\right) du\right)
-\sum\limits_{k\in 
\mathbb{Z}
}\chi \left( wx-t_{k},f\left( x\right) \right) \right\vert \right) dx\right.
\\
&&\left. +\int\limits_{%
\mathbb{R}
}\varphi \left( 3\mu \left\vert \sum\limits_{k\in 
\mathbb{Z}
}\chi \left( wx-t_{k},f\left( x\right) \right) -f\left( x\right) \right\vert
\right) dx\right\} =:I_{1}+I_{2}+I_{3.}
\end{eqnarray*}%
Now, we estimate $I_{1}.$ Applying condition $(\chi 3)$, we have%
\begin{equation*}
3I_{1}=\int\limits_{%
\mathbb{R}
}\varphi \left( 3\mu \left\vert \left( S_{w}f\right) \left( x\right)
-\sum\limits_{k\in 
\mathbb{Z}
}\chi \left( wx-t_{k},\frac{w}{\Delta _{k}} \hskip-.3cm  \int%
\limits_{t_{k}/w}^{t_{k+1}/w}\hskip-.3cm  f\left( u+x-t_{k}/w\right) du\right)
\right\vert \right) dx
\end{equation*}%
\begin{equation*}
\leq \int\limits_{%
\mathbb{R}
}\varphi \!\left( 3\mu \sum\limits_{k\in 
\mathbb{Z}
}\left\vert \chi\! \left( wx-t_{k},\frac{w}{\Delta _{k}} \hskip-.3cm \int%
\limits_{t_{k}/w}^{t_{k+1}/w}
\hskip-.3cm 
f\left( u\right) du\right)\! -\!\chi \left(\!
wx-t_{k},\frac{w}{\Delta _{k}}\hskip-.3cm  \int\limits_{t_{k}/w}^{t_{k+1}/w}
\hskip-.3cm 
f\left(
u+x-t_{k}/w\right) du\right) \right\vert \right)\! dx{{}}
\end{equation*}%
\begin{eqnarray*}
&\leq &\int\limits_{%
\mathbb{R}
}\varphi \left( 3\mu \sum\limits_{k\in 
\mathbb{Z}
}L\left( wx-t_{k}\right) \psi \left( \left\vert \frac{w}{\Delta _{k}}\hskip-.3cm  
\int\limits_{t_{k}/w}^{t_{k+1}/w}
\hskip-.3cm 
\left[ f\left( u\right) -f\left(
u+x-t_{k}/w\right) \right] du\right\vert \right) \right) dx \\
&\leq &\int\limits_{%
\mathbb{R}
}\varphi \left( 3\mu \sum\limits_{k\in 
\mathbb{Z}
}L\left( wx-t_{k}\right) \psi \left( \frac{w}{\Delta _{k}}\hskip-.3cm 
\int\limits_{t_{k}/w}^{t_{k+1}/w}
\hskip-.3cm 
\left\vert f\left( u\right) -f\left(
u+x-t_{k}/w\right) \right\vert du\right) \right) dx.
\end{eqnarray*}%
Applying Jensen inequality twice (see e.g., \cite{COSP11}), the change of variable $y=x-t_{k}/w,$
condition (H) and Fubini-Tonelli theorem, we obtain%
\vskip-1cm
\begin{eqnarray*}
&& 3I_{1} \leq \\ &\leq&\frac{1}{m_{0,\Pi }\left( L\right) }\int\limits_{%
\mathbb{R}
}\sum\limits_{k\in 
\mathbb{Z}
}L\left( wx-t_{k}\right) \varphi \left( 3\mu m_{0,\Pi }\left( L\right) \psi
\left( \frac{w}{\Delta _{k}}  \hskip-.3cm \int\limits_{t_{k}/w}^{t_{k+1}/w}  \hskip-.3cm \left\vert
f\left( u\right) -f\left( u+x-t_{k}/w\right) \right\vert du\right) \right) dx
\\
&\leq &\!\!\!\!\frac{1}{m_{0,\Pi }\left( L\right) }\sum\limits_{k\in 
\mathbb{Z}
}\int\limits_{%
\mathbb{R}
}L\left( wx-t_{k}\right) \varphi \left( 3\mu m_{0,\Pi }\left( L\right) \psi
\left( \frac{w}{\Delta _{k}} \hskip-.3cm \int\limits_{t_{k}/w}^{t_{k+1}/w}  \hskip-.3cm \left\vert
f\left( u\right) -f\left( u+x-t_{k}/w\right) \right\vert du\right) \right) dx
\\
&\leq &\!\!\!\frac{1}{m_{0,\Pi }\left( L\right) }\sum\limits_{k\in 
\mathbb{Z}
}\int\limits_{%
\mathbb{R}
}L\left( wy\right) \varphi \left( 3\mu m_{0,\Pi }\left( L\right) \psi \left( 
\frac{w}{\Delta _{k}}  \hskip-.3cm \int\limits_{t_{k}/w}^{t_{k+1}/w}
 \hskip-.3cm
\left\vert f\left(
u\right) -f\left( u+y\right) \right\vert du\right) \right) dy
\\
&\leq &\frac{1}{m_{0,\Pi }\left( L\right) }\sum\limits_{k\in 
\mathbb{Z}
}\int\limits_{%
\mathbb{R}
}L\left( wy\right) \varphi \left( C_{\lambda }\psi \left( \frac{w}{\Delta
_{k}}  \hskip-.3cm \int\limits_{t_{k}/w}^{t_{k+1}/w}  \hskip-.3cm \left\vert f\left( u\right) -f\left(
u+y\right) \right\vert du\right) \right) dy \\
&\leq &\frac{1}{m_{0,\Pi }\left( L\right) }\sum\limits_{k\in 
\mathbb{Z}
}\int\limits_{%
\mathbb{R}
}L\left( wy\right) \eta \left( \lambda \frac{w}{\Delta _{k}} 
 \hskip-.3cm \int\limits_{t_{k}/w}^{t_{k+1}/w}  \hskip-.3cm \left\vert f\left( u\right) -f\left(
u+y\right) \right\vert du\right) dy \\
&\leq &\frac{1}{m_{0,\Pi }\left( L\right) }\sum\limits_{k\in 
\mathbb{Z}
}\int\limits_{%
\mathbb{R}
}L\left( wy\right) \frac{w}{\Delta _{k}}  \hskip-.3cm \int\limits_{t_{k}/w}^{t_{k+1}/w}%
\eta \left( \lambda \left\vert f\left( u\right) -f\left( u+y\right)
\right\vert \right) dudy \\
&\leq &\frac{\delta ^{-1}}{m_{0,\Pi }\left( L\right) }\int\limits_{%
\mathbb{R}
}wL\left( wy\right) \sum\limits_{k\in 
\mathbb{Z}
}\int\limits_{t_{k}/w}^{t_{k+1}/w}\eta \left( \lambda \left\vert f\left(
u\right) -f\left( u+y\right) \right\vert \right) dudy \\
&\leq &\frac{\delta ^{-1}}{m_{0,\Pi }\left( L\right) }\int\limits_{%
\mathbb{R}
}wL\left( wy\right) \left[ \int\limits_{%
\mathbb{R}
}\eta \left( \lambda \left\vert f\left( u\right) -f\left( u+y\right)
\right\vert \right) du\right] dy \\
&\leq &\frac{\delta ^{-1}}{m_{0,\Pi }\left( L\right) }\int\limits_{%
\mathbb{R}
}wL\left( wy\right) I^{\eta }\left[ \lambda \left( f\left( \cdot \right)
-f\left( \cdot +y\right) \right) \right] dy=:J.
\end{eqnarray*}
Now, let $0<\alpha <1$ be fixed. We now split the above integral J as%
\begin{eqnarray*}
J :=\frac{w\delta ^{-1}}{m_{0,\Pi }\left( L\right) }\left\{
\int\limits_{\left\vert y\right\vert \leq 1/w^{\alpha
}}+\int\limits_{\left\vert y\right\vert >1/w^{\alpha }}\right\} L\left(
wy\right) I^{\eta }\left[ \lambda \left( f\left( \cdot \right) -f\left(
\cdot +y\right) \right) \right] dy 
=:J_{1}+J_{2}.
\end{eqnarray*}%
For $J_{1},$ one has%
\begin{eqnarray*}
J_{1} &=&\frac{w\delta ^{-1}}{m_{0,\Pi }\left( L\right) }\int\limits_{\left%
\vert y\right\vert \leq 1/w^{\alpha }}L\left( wy\right) I^{\eta }\left[
\lambda \left( f\left( \cdot \right) -f\left( \cdot +y\right) \right) \right]
dy \\
&\leq& \frac{w\delta ^{-1}}{m_{0,\Pi }\left( L\right) }\int\limits_{\left%
\vert y\right\vert \leq 1/w^{\alpha }}L\left( wy\right) \omega \left(
f,\left\vert y\right\vert \right) _{\eta }dy 
\leq \omega \left( f,1/w^{\alpha }\right) _{\eta }\frac{w\delta ^{-1}}{%
m_{0,\Pi }\left( L\right) }\int\limits_{\left\vert y\right\vert \leq
1/w^{\alpha }}L\left( wy\right) dy \\
&\leq &\omega \left( f,1/w^{\alpha }\right) _{\eta }\frac{\delta ^{-1}}{%
m_{0,\Pi }\left( L\right) }\left\Vert L\right\Vert _{1}.
\end{eqnarray*}%
On the other hand, taking into account that $I^{\eta }$ is convex (since $%
\eta $ is convex), for $J_{2}$ we can write%
\vskip-1cm
\begin{eqnarray*}
J_{2} &=&\frac{w\delta ^{-1}}{m_{0,\Pi }\left( L\right) }\int\limits_{\left%
\vert y\right\vert >1/w^{\alpha }}L\left( wy\right) I^{\eta }\left[ \lambda
\left( f\left( \cdot \right) -f\left( \cdot +y\right) \right) \right] dy \\
&\leq &\frac{w\delta ^{-1}}{m_{0,\Pi }\left( L\right) }\int\limits_{\left%
\vert y\right\vert >1/w^{\alpha }}L\left( wy\right) \frac{1}{2}\left\{
I^{\eta }\left[ 2\lambda f\left( \cdot \right) \right] +I^{\eta }\left[
2\lambda f\left( \cdot +y\right) \right] \right\} dy.
\end{eqnarray*}%
\vskip-0.11cm
Moreover, it can be easily seen that%
\begin{equation*}
I^{\eta }\left[ 2\lambda f\left( \cdot \right) \right] =I^{\eta }\left[
2\lambda f\left( \cdot +y\right) \right] ,
\end{equation*}%
for every $y.$ Therefore, by assumption (\ref{a1}), we have%
\begin{eqnarray*}
J_{2} &\leq &\frac{w\delta ^{-1}}{m_{0,\Pi }\left( L\right) }%
\int\limits_{\left\vert y\right\vert >1/w^{\alpha }}L\left( wy\right)
I^{\eta }\left[ 2\lambda f\right] dy
\leq \frac{\delta ^{-1}}{m_{0,\Pi }\left( L\right) }I^{\eta }\left[
\lambda _{0}f\right] M_{1}w^{-\alpha _{0}},\text{ as }w\rightarrow +\infty .
\end{eqnarray*}%
Now we estimate $I_{2}.$%
\begin{eqnarray*}
3I_{2} &=&\int\limits_{%
\mathbb{R}
}\varphi \left( 3\mu \left\vert \sum\limits_{k\in 
\mathbb{Z}
}\chi \left( wx-t_{k},\frac{w}{\Delta _{k}} \hskip-.3cm \int%
\limits_{t_{k}/w}^{t_{k+1}/w}
\hskip-.3cm 
f\left( u+x-t_{k}/w\right) du\right)
-\sum\limits_{k\in 
\mathbb{Z}
}\chi \left( wx-t_{k},f\left( x\right) \right) \right\vert \right) dx \\
&\leq &\int\limits_{%
\mathbb{R}
}\varphi \left( 3\mu \sum\limits_{k\in 
\mathbb{Z}
}\left\vert \chi \left( wx-t_{k},\frac{w}{\Delta _{k}}\hskip-.3cm  \int%
\limits_{t_{k}/w}^{t_{k+1}/w} \hskip-.3cm  f\left( u+x-t_{k}/w\right) du\right) -\chi
\left( wx-t_{k},f\left( x\right) \right) \right\vert \right) dx.
\end{eqnarray*}%
Using condition $(\chi 3)$, the change of variable $y=u-t_{k}/w,$
Jensen inequality and condition (H), we have
$$
3I_{2} \leq \int\limits_{%
\mathbb{R}
}\varphi \left( 3\mu \sum\limits_{k\in 
\mathbb{Z}
}L\left( wx-t_{k}\right) \psi \left( \left\vert \frac{w}{\Delta _{k}}%
\int\limits_{t_{k}/w}^{t_{k+1}/w}f\left( u+x-t_{k}/w\right) du-f\left(
x\right) \right\vert \right) \right) dx 
$$
\begin{eqnarray*}
&\leq &\int\limits_{%
\mathbb{R}
}\varphi \left( 3\mu \sum\limits_{k\in 
\mathbb{Z}
}L\left( wx-t_{k}\right) \psi \left( \left\vert \frac{w}{\Delta _{k}}%
\int\limits_{0}^{\Delta _{k}/w}\left[ f\left( x+y\right) -f\left( x\right) %
\right] dy\right\vert \right) \right) dx \\
&\leq &\int\limits_{%
\mathbb{R}
}\varphi \left( 3\mu \sum\limits_{k\in 
\mathbb{Z}
}L\left( wx-t_{k}\right) \psi \left( \frac{w}{\Delta _{k}}%
\int\limits_{0}^{\Delta _{k}/w}\left\vert f\left( x+y\right) -f\left(
x\right) \right\vert dy\right) \right) dx 
\\
&\leq &\frac{1}{m_{0,\Pi }\left( L\right) }\int\limits_{%
\mathbb{R}
}\sum\limits_{k\in 
\mathbb{Z}
}L\left( wx-t_{k}\right) \varphi \left( 3\mu m_{0,\Pi }\left( L\right) \psi
\left( \frac{w}{\Delta _{k}} \hskip-.3cm \int\limits_{0}^{\Delta _{k}/w} \hskip-.3cm \left\vert
f\left( x+y\right) -f\left( x\right) \right\vert dy\right) \right) dx \\
&\leq &\frac{1}{m_{0,\Pi }\left( L\right) }\int\limits_{%
\mathbb{R}
}\sum\limits_{k\in 
\mathbb{Z}
}L\left( wx-t_{k}\right) \varphi \left( C_{\lambda }\psi \left( \frac{w}{%
\Delta _{k}}\int\limits_{0}^{\Delta _{k}/w}\left\vert f\left( x+y\right)
-f\left( x\right) \right\vert dy\right) \right) dx \\
&\leq &\frac{1}{m_{0,\Pi }\left( L\right) }\int\limits_{%
\mathbb{R}
}\sum\limits_{k\in 
\mathbb{Z}
}L\left( wx-t_{k}\right) \eta \left( \lambda \frac{w}{\Delta _{k}}%
\int\limits_{0}^{\Delta _{k}/w}\left\vert f\left( x+y\right) -f\left(
x\right) \right\vert dy\right) dx.
\end{eqnarray*}%
Using Jensen inequality and Fubini-Tonelli theorem, we get%
\begin{eqnarray*}
3I_{2} &\leq& \frac{1}{m_{0,\Pi }\left( L\right) }\int\limits_{%
\mathbb{R}
}\sum\limits_{k\in 
\mathbb{Z}
}L\left( wx-t_{k}\right) \left[ \frac{w}{\Delta _{k}}\int\limits_{0}^{%
\Delta _{k}/w}\eta \left( \lambda \left\vert f\left( x+y\right) -f\left(
x\right) \right\vert \right) dy\right] dx 
\\
&\leq& \frac{\delta ^{-1}}{m_{0,\Pi }\left( L\right) }\int\limits_{%
\mathbb{R}
}w\sum\limits_{k\in 
\mathbb{Z}
}L\left( wx-t_{k}\right) \left[ \int\limits_{0}^{\Delta /w}\eta \left(
\lambda \left\vert f\left( x+y\right) -f\left( x\right) \right\vert \right)
dy\right] dx 
\\
&\leq& \frac{\delta ^{-1}}{m_{0,\Pi }\left( L\right) }\int\limits_{%
\mathbb{R}
}wm_{0,\Pi }\left( L\right) \left[ \int\limits_{0}^{\Delta /w}\eta \left(
\lambda \left\vert f\left( x+y\right) -f\left( x\right) \right\vert \right)
dy\right] dx 
\\
&=&\delta ^{-1}w\int\limits_{0}^{\Delta /w}\left[ \int\limits_{%
\mathbb{R}
}\eta \left( \lambda \left\vert f\left( x+y\right) -f\left( x\right)
\right\vert \right) dx\right] dy\\
&=& \delta ^{-1}w\int\limits_{0}^{\Delta /w}I^{\eta }\left[ \lambda \left(
f\left( \cdot +y\right) -f\left( \cdot \right) \right) \right] dy 
\leq \delta ^{-1}\omega \left( f,\Delta /w\right) _{\eta
}w\int\limits_{0}^{\Delta /w}dy \\
&=& \delta ^{-1}\Delta \omega \left( f,\Delta
/w\right) _{\eta }.
\end{eqnarray*}
For $I_{3},$ denoted by $A_{0}\subseteq 
\mathbb{R}
$ the set of all points of $
\mathbb{R}
$ for which $f\neq 0$ almost everywhere, we obtain%
\begin{eqnarray*}
3I_{3} &=&\int\limits_{A_{0}}\varphi \left( 3\mu \left\vert
\sum\limits_{k\in \mathbb{Z}}
\chi \left( wx-t_{k},f\left( x\right) \right) -f\left( x\right) \right\vert
\right) dx \\
&=&\int\limits_{A_{0}}\varphi \left( 3\mu \left\vert f\left( x\right)
\right\vert \left\vert \frac{1}{f\left( x\right) }\sum\limits_{k\in 
\mathbb{Z}
}\chi \left( wx-t_{k},f\left( x\right) \right) -1\right\vert \right) dx \\
&\leq &\int\limits_{A_{0}}\varphi \left( 3\mu \left\vert f\left( x\right)
\right\vert \mathcal{T}_{w}\left( x\right) \right) dx.
\end{eqnarray*}%
By the convexity of $\varphi $ and condition $(\chi 4)$, we have%
\begin{eqnarray*}
3I_{3} &\leq &\int\limits_{A_{0}}\varphi \left( 3\mu M_{2}w^{-\theta
_{0}}\left\vert f\left( x\right) \right\vert \right) dx 
\leq w^{-\theta _{0}}\int\limits_{A_{0}}\varphi \left( 3M_{2}\mu
\left\vert f\left( x\right) \right\vert \right) dx \\
&\leq &w^{-\theta _{0}}\int\limits_{%
\mathbb{R}
}\varphi \left( 3M_{2}\mu \left\vert f\left( x\right) \right\vert \right) dx
\leq w^{-\theta _{0}}I^{\varphi }\left[ 3M_{2}\mu f\right] \leq w^{-\theta
_{0}}I^{\varphi }\left[ \lambda _{0}f\right] <+\infty,
\end{eqnarray*}%
for positive constants$\ M_{2},$ $\theta _{0}$. This completes the proof.
\end{proof}
\end{theorem}

Note that, condition (\ref{a1}) is satisfied when, for instance, the kernel $%
\chi $ satisfies condition $(\chi 3)$ with $L$ having compact support,
e.g. $supp$ $L\subset \left[ -B,B\right] ,$ $B>0.$ Indeed,%
\begin{equation*}
w\int \limits_{\left \vert y\right \vert >1/w^{\alpha }}L\left( wy\right)
dy=w\int \limits_{\left \vert u\right \vert >w^{1-\alpha }}L\left( u\right)
du=0,
\end{equation*}%
for sufficiently large $w>B^{1/\left( 1-\alpha \right) }.$ Moreover, in this case, condition $(L2)$ is satisfied for every $%
\beta _{0}>0.$ Then, we get the following.

\begin{corollary}
\label{Cr}Let $\chi $ be a kernel satisfying condition $(\chi 3)$ with $L$
having compact support. Moreover, let $\varphi $ be a convex $\varphi -$%
function satisfying condition (H) with $\eta $ convex and $f\in L^{\varphi
+\eta }$\bigskip $\left( 
\mathbb{R}
\right) .$ Then, 
for every 
$0<\alpha <1,$ there exist
constant $\mu >0$ and $\lambda_0 >0$ such that%
\begin{equation*}
I^{\varphi }\left[ \mu \left( S_{w}f-f\right) \right] \leq \frac{\left\Vert
L\right\Vert _{1}}{3\delta m_{0,\Pi }\left( L\right) }\omega \left(
f,1/w^{\alpha }\right) _{\eta }+\frac{\Delta }{3\delta }\omega \left(
f,\Delta /w\right) _{\eta }+\frac{I^{\varphi }\left[ \lambda _{0}f\right] }{3%
}w^{-\theta _{0}},
\end{equation*}%
for every sufficiently large $w>0,$ where $m_{0,\Pi }\left( L\right)
<+\infty $ in view of Lemma \ref{Lem1} and $\theta _{0}>0$ is the constant
of condition $(\chi 4).$
\end{corollary}
\vskip0.4cm

Note that, if $L$ has not compact support, we may require the following sufficient condition:%
\begin{equation}
\label{a5}
M_\nu(L)\ :=\ \int_{\mathbb{R}}L(u)\, |u|^\nu\, du\ <\ +\infty, \quad \quad \nu>0,
\end{equation}
which imply assumption (\ref{a1}). In this case, for every $0<\alpha <1,$
we get
$$
w\int\limits_{\left\vert y\right\vert >1/w^{\alpha }}L\left( wy\right)
dy\ =\ \int_{|u|>w^{1-\alpha}} L(u)\, du\ = \int_{|u|>w^{1-\alpha}} \frac{|u|^\nu}{|u|^\nu}L(u)\, du
$$
$$
\leq\ {1 \over w^{(1-\alpha)\, \nu}}\int_{|u|>w^{1-\alpha}} |u|^\nu L(u)\, du\ \leq\ {M_\nu(L) \over w^{(1-\alpha)\, \nu}} \ =\ \mathcal{O}\left( w^{(\alpha -1)\, \nu}\right), 
$$
for sufficiently large $w>0$. Hence, (\ref{a1}) is satisfied with $\alpha_0=(1-\alpha)\, \nu$ and $M_1=M_\nu(L)$.

  We now recall the definition of Lipschitz classes in Orlicz spaces. We define by $Lip_{\varphi}(\nu)$, $0<\nu \leq 1$, the set of all functions $f \in M(\mathbb{R})$ such that, there exists $\lambda>0$ with:
\begin{equation*}
I^{\varphi}[\lambda(f(\cdot)-f(\cdot+t))] = \int_{\mathbb{R}}\varphi\left(\lambda\left|f\left(x\right)-f\left(x + t\right)\right|\right) dx = \mathcal{O}(|t|^{\nu}),
\end{equation*}
as $t\to 0$. In this context, from Theorem \ref{Th1} we immediately get the following corollary. 
\begin{corollary} \label{cor-Lip-varphi}
Under the assumptions of Theorem \ref{Th1} with $0<\alpha < 1$, and for any $f \in Lip_{\eta}(\nu)$, $0<\nu \leq 1$, there exist $K>0$, $\mu>0$ such that
$$
I^{\varphi}[\mu(S_wf - f)] \leq K\, w^{-\ell},
$$
for  sufficiently large $w>0$, where $\ell:=\min\left\{ \alpha \nu,\, \alpha_0,\, \theta_0  \right\}$.
\end{corollary}
\hskip0.4cm

Now, we consider below some particular cases of Orlicz spaces.

\noindent Let $\varphi \left( u\right) =u^{p},$ $u\in 
\mathbb{R}
_{0}^{+},$ $p\geq 1.$ Then, the Orlicz space $L^{\varphi }\left( 
\mathbb{R}
\right) $ coincides with the space $L^{p}\left( 
\mathbb{R}
\right) .$ If 
$\psi \left( u\right) =u^{q/p},$ $1 \leq q \le p$, condition (H) it turns out to be satisfied with $\eta \left( u\right)=u^q$ and $C_{\lambda }=\lambda^{q/p}$. In this case, we have $L^{\varphi + \eta}(\mathbb{R})= L^p(\mathbb{R}) \cap L^q(\mathbb{R})$ (which is, in fact, a proper subspace of $L^p(\mathbb{R})$), and obviously Theorem \ref{Th1} and its corollary hold.

  From the theory developed in \cite{Vinti}, we know that if the function $\psi$ of condition $(\chi 3)$ is $\psi(u)=u$, $u \in \mathbb{R}$, and $\varphi(u)=u^p$, $1 \leq p <\infty$, the operators $S_w$ maps the whole space $L^p(\mathbb{R})$ into itself and therefore, we can obtain, as particular case, a quantitative estimate in $L^p(\mathbb{R})$. But since the $\varphi$-modulus of continuity does not satisfy the well-known property $\omega(f, \lambda \delta) \leq (1 + \lambda)\omega(f,\delta)$, satisfied e.g., by the $\omega_p$-modulus of smoothness below defined, we can proceed using a direct approach and estimating the term $S_w f - f$ with respect to the p-norm. For the above purpose, we need to recall the definition of the $L^{p}-$modulus of smoothness of order one
given as%
\begin{equation*}
\omega _{p}\left( f,\delta \right) :=\sup_{\left \vert h\right \vert \leq
\delta }\left \Vert f\left( \cdot +h\right) -f\left( \cdot \right) \right
\Vert _{p}=\sup_{\left \vert h\right \vert \leq \delta }\left( \int
\limits_{%
\mathbb{R}
}\left \vert f\left( t+h\right) -f\left( t\right) \right \vert ^{p}dt\right)
^{1/p},
\end{equation*}%
with $\delta >0,$ $f\in L^{p}\left( 
\mathbb{R}
\right) ,$ $1\leq p<\infty .$

We can prove the following estimate.
\begin{theorem} \label{theorem-lp-2}
Suppose that $M_p(L)<+\infty$, $1\leq p<\infty$.
Then, for every $f\in L^{p}\left(\mathbb{R}\right),$ the following quantitative estimate holds
\begin{eqnarray*}
\left \Vert S_{w}f-f\right \Vert _{p} &\leq& \delta ^{-1/p} [2m_{0,\Pi }\left(
	L\right)]^{\left( p-1\right) /p}\left[ \left \Vert
L\right \Vert _{1}+M_{p}\left( L\right) \right] ^{1/p}\omega _{p}\left(f,1/w\right)
\\
&+& \delta ^{-1/p}m_{0,\Pi }\left( L\right) \Delta ^{1/p}\omega _{p}\left(f,\Delta /w\right) +M_{2}\left \Vert f\right \Vert _{p}w^{-\theta _{0}},
\end{eqnarray*}
for sufficiently large $w>0,$ where $m_{0,\Pi }\left( L\right)<+\infty $ in view of Lemma \ref{Lem1} and $M_{2},$ $\theta _{0}>0$ are the
constants of condition $(\chi 4)$.
\end{theorem}
\begin{proof}
Recalling that $I^\varphi[f]=\|f\|_p^p$, when $\varphi(u)=u^p$, proceeding as in the first part of the proof of Theorem \ref{Th1}, using the Minkowsky inequality, and that the function $|\cdot|^{1/p}$ is concave and hence sub-additive, we immediately obtain:
\begin{eqnarray*}
\left\Vert S_{w}f-f\right\Vert _{p} &\leq& \left( \int\limits_{%
\mathbb{R}
}\left[ \sum\limits_{k\in 
\mathbb{Z}
}L\left( wx-t_{k}\right) \frac{w}{\Delta _{k}}\int%
\limits_{t_{k}/w}^{t_{k+1}/w}\left\vert f\left( u\right) -f\left(
u+x-t_{k}/w\right) \right\vert du\right] ^{p}dx\right) ^{1/p} 
\\
&+& \left( \int\limits_{%
\mathbb{R}
}\left[ \sum\limits_{k\in 
\mathbb{Z}
}L\left( wx-t_{k}\right) \frac{w}{\Delta _{k}}\int%
\limits_{t_{k}/w}^{t_{k+1}/w}\left\vert f\left( u+x-t_{k}/w\right) -f\left(
x\right) \right\vert du\right] ^{p}dx\right) ^{1/p} 
\\
&+& \left( \int\limits_{%
\mathbb{R}
}\left\vert \sum\limits_{k\in 
\mathbb{Z}
}\chi \left( wx-t_{k},f\left( x\right) \right) -f\left( x\right) \right\vert
^{p}dx\right) ^{1/p}=:I_{1}+I_{2}+I_{3.}
\end{eqnarray*}
We estimate $I_{1}.$ Applying Jensen inequality twice, Fubini-Tonelli theorem, and by the change of
variable $y=x-t_{k}/w$, we get%
\begin{eqnarray*}
I_{1}^{p} &=&\int\limits_{%
\mathbb{R}
}\left[ \sum\limits_{k\in 
\mathbb{Z}
}L\left( wx-t_{k}\right) \frac{w}{\Delta _{k}}\int%
\limits_{t_{k}/w}^{t_{k+1}/w}\left\vert f\left( u\right) -f\left(
u+x-t_{k}/w\right) \right\vert du\right] ^{p}dx \\
&\leq &\frac{1}{m_{0,\Pi }\left( L\right) }\int\limits_{%
\mathbb{R}
}\sum\limits_{k\in 
\mathbb{Z}
}L\left( wx-t_{k}\right) \left[ \frac{w}{\Delta _{k}}\int%
\limits_{t_{k}/w}^{t_{k+1}/w}m_{0,\Pi }\left( L\right) \left\vert f\left(
u\right) -f\left( u+x-t_{k}/w\right) \right\vert du\right] ^{p}dx \\
&\leq &m_{0,\Pi }\left( L\right) ^{p-1}\sum\limits_{k\in 
\mathbb{Z}
} \int\limits_{%
\mathbb{R}
}L\left( wx-t_{k}\right) \left[ \frac{w}{\Delta _{k}}\int%
\limits_{t_{k}/w}^{t_{k+1}/w}\left\vert f\left( u\right) -f\left(
u+x-t_{k}/w\right) \right\vert ^{p}du\right] dx \\
&\leq &m_{0,\Pi }\left( L\right) ^{p-1}\sum\limits_{k\in 
\mathbb{Z}
} \int\limits_{%
\mathbb{R}
}L\left( wy\right) \left[ \frac{w}{\Delta _{k}}\int%
\limits_{t_{k}/w}^{t_{k+1}/w}\left\vert f\left( u\right) -f\left( u+y\right)
\right\vert ^{p}du\right] dy \\
&\leq &\delta ^{-1}m_{0,\Pi }\left( L\right) ^{p-1}\int\limits_{%
\mathbb{R}
}wL\left( wy\right) \left[ \sum\limits_{k\in 
\mathbb{Z}
}\int\limits_{t_{k}/w}^{t_{k+1}/w}\left\vert f\left( u\right) -f\left(
u+y\right) \right\vert ^{p}du\right] dy \\
&=&\delta ^{-1}m_{0,\Pi }\left( L\right) ^{p-1}\int\limits_{%
\mathbb{R}
}wL\left( wy\right) \left[ \int\limits_{%
\mathbb{R}
}\left\vert f\left( u\right) -f\left( u+y\right) \right\vert ^{p}du\right] dy
\\
&\leq &\delta ^{-1}m_{0,\Pi }\left( L\right) ^{p-1}\int\limits_{%
\mathbb{R}
}wL\left( wy\right) \omega _{p}\left( f,\left\vert y\right\vert \right)
^{p}dy \\
&\leq &\delta ^{-1}m_{0,\Pi }\left( L\right) ^{p-1}\int\limits_{%
\mathbb{R}
}wL\left( wy\right) \left( 1+w\left\vert y\right\vert \right) ^{p}\omega
_{p}\left( f,\frac{1}{w}\right) ^{p}dy 
\end{eqnarray*}
\begin{eqnarray*}
&\leq &\delta ^{-1}m_{0,\Pi }\left( L\right) ^{p-1}2^{p-1}\omega _{p}\left(
f,\frac{1}{w}\right) ^{p}\int\limits_{%
\mathbb{R}
}wL\left( wy\right) \left( 1+\left( w\left\vert y\right\vert \right)
^{p}\right) dy \\
&=&\delta ^{-1}m_{0,\Pi }\left( L\right) ^{p-1}2^{p-1}\omega _{p}\left( f,%
\frac{1}{w}\right) ^{p}\int\limits_{%
\mathbb{R}
}L\left( z\right) \left( 1+\left\vert z\right\vert ^{p}\right) dz \\
&=&\delta ^{-1}m_{0,\Pi }\left( L\right) ^{p-1}2^{p-1}\omega _{p}\left( f,%
\frac{1}{w}\right) ^{p}\left( \left\Vert L\right\Vert _{1}+M_{p}\left(
L\right) \right) ,
\end{eqnarray*}%
for every $w>0,$ where $\left\Vert L\right\Vert _{1}$ and $M_{p}\left(
L\right) $ are both finite. 
Note that, in the above estimates we used the following well-known property of the modulus of smoothness
$$
\omega_p(f, \lambda \delta)\ \leq\ (1 + \lambda)\, \omega_p(f, \delta), \quad \quad \lambda,\ \delta>0.
$$
Now we estimate $I_{2}.$ Using Jensen inequality twice, the change of
variable $y=u-t_{k}/w$ and Fubini-Tonelli theorem, we have%
\begin{eqnarray*}
I_{2}^{p} &=&\int\limits_{%
\mathbb{R}
}\left[ \sum\limits_{k\in 
\mathbb{Z}
}L\left( wx-t_{k}\right) \frac{w}{\Delta _{k}}\int%
\limits_{t_{k}/w}^{t_{k+1}/w}\left\vert f\left( u+x-t_{k}/w\right) -f\left(
x\right) \right\vert du\right] ^{p}dx \\
&\leq &\int\limits_{%
\mathbb{R}
}\left[ \sum\limits_{k\in 
\mathbb{Z}
}L\left( wx-t_{k}\right) \frac{w}{\Delta _{k}}\int\limits_{0}^{\Delta
_{k}/w}\left\vert f\left( x+y\right) -f\left( x\right) \right\vert dy\right]
^{p}dx \\
&\leq &\frac{1}{m_{0,\Pi }\left( L\right) }\int\limits_{%
\mathbb{R}
}\sum\limits_{k\in 
\mathbb{Z}
}L\left( wx-t_{k}\right) \left[ \frac{w}{\Delta _{k}}\int\limits_{0}^{%
\Delta _{k}/w}m_{0,\Pi }\left( L\right) \left\vert f\left( x+y\right)
-f\left( x\right) \right\vert dy\right] ^{p}dx \\
&\leq &\delta ^{-1}m_{0,\Pi }\left( L\right) ^{p-1}\int\limits_{%
\mathbb{R}
}\sum\limits_{k\in 
\mathbb{Z}
}L\left( wx-t_{k}\right) \left[ w\int\limits_{0}^{\Delta /w}\left\vert
f\left( x+y\right) -f\left( x\right) \right\vert ^{p}dy\right] dx \\
&\leq &\delta ^{-1}m_{0,\Pi }\left( L\right) ^{p}\int\limits_{%
\mathbb{R}
}w\int\limits_{0}^{\Delta /w}\left\vert f\left( x+y\right) -f\left(
x\right) \right\vert ^{p}dydx \\
&\leq &\delta ^{-1}m_{0,\Pi }\left( L\right) ^{p}\int\limits_{0}^{\Delta
/w}w\left[ \int\limits_{%
\mathbb{R}
}\left\vert f\left( x+y\right) -f\left( x\right) \right\vert ^{p}dx\right] dy
\\
&\leq &\delta ^{-1}m_{0,\Pi }\left( L\right) ^{p}\int\limits_{0}^{\Delta
/w}w\left[ \omega _{p}\left( f,\Delta /w\right) \right] ^{p}dy 
= \delta ^{-1}m_{0,\Pi }\left( L\right) ^{p}\Delta \left[ \omega _{p}\left(
f,\Delta /w\right) \right] ^{p}.
\end{eqnarray*}%
For $I_{3},$ denoted by $A_{0}\subseteq 
\mathbb{R}
$ the set of all points of $%
\mathbb{R}
$ for which $f\neq 0$ almost everywhere, we obtain%
\begin{eqnarray*}
I_{3}^{p} &=&\int\limits_{A_{0}}\left\vert \sum\limits_{k\in 
\mathbb{Z}
}\chi \left( wx-t_{k},f\left( x\right) \right) -f\left( x\right) \right\vert
^{p}dx \\
&=&\int\limits_{A_{0}}\left\vert f\left( x\right) \right\vert
^{p}\left\vert \frac{1}{f\left( x\right) }\sum\limits_{k\in 
\mathbb{Z}
}\chi \left( wx-t_{k},f\left( x\right) \right) -1\right\vert ^{p}dx 
\leq \int\limits_{A_{0}}\left\vert f\left( x\right) \right\vert ^{p}\left[
\mathcal{T}_{w}\left( x\right) \right] ^{p}dx.
\end{eqnarray*}%
From condition $(\chi 4)$, we have%
\begin{eqnarray*}
I_{3}^{p} &\leq &\int\limits_{A_{0}}\left\vert f\left( x\right) \right\vert
^{p}M_{2}^{p}w^{-p\theta _{0}}dx 
\leq M_{2}^{p}w^{-p\theta _{0}}\int\limits_{%
\mathbb{R}
}\left\vert f\left( x\right) \right\vert ^{p}dx \\
&=&M_{2}^{p}w^{-p\theta _{0}}\left\Vert f\right\Vert _{p}^{p}
\end{eqnarray*}%
for positive constants$\ M_{2},$ $\theta _{0}$. This proves the theorem.
\end{proof}

Now, denoting by $Lip\left( \alpha ,p\right)$, $0<\alpha \leq 1$, $p\geq 1$, the corresponding Lipschitz classes in $L^p(\mathbb{R})$, we can immediately state the following.

\begin{corollary} \label{cor3.3}
Suppose that $M_p(L)<+\infty$, for $1 \leq p<\infty$.
Then, for every $f\in Lip\left( \alpha ,p\right) ,$ with $0<\alpha \leq 1$, we have%
\begin{eqnarray*}
\left\Vert S_{w}f-f\right\Vert _{p} &\leq &\delta ^{-1/p}[2m_{0,\Pi }\left(
L\right)]^{\left( p-1\right) /p}\left[  \|L\|_1+ M_{p}\left( L\right) \right]^{1/p}C_{1} \frac{%
1}{w^{\alpha }} \\
&&+\delta ^{-1/p}m_{0,\Pi }\left( L\right) C_{1}\Delta ^{1/p}\left( \frac{%
\Delta }{w}\right) ^{\alpha }+M_{2}\left\Vert f\right\Vert _{p}w^{-\theta
_{0}},
\end{eqnarray*}%
for every sufficiently large $w>0,$ where $m_{0,\Pi }\left( L\right) $ is finite in view of Lemma \ref{Lem1} and, $C_{1}>0$, $M_{2},$ $\theta _{0}>0$ are the constants arising from the fact that $f \in Lip(\alpha, p)$ and from
condition $(\chi 4),$ respectively$.$
\end{corollary}

\begin{remark} \rm
\label{Rk}Note that if the kernel $\chi $ is of the form $\chi \left(
x,u\right) =L\left( x\right) u,$ with $L$ satisfying conditions $\left(
L1\right) $ and $\left( L2\right) ,$ the above operators reduces to the linear
case considered in \cite{Bardaro2}. In this situation, condition $\left( \chi
4\right) $ becomes%
\begin{equation}
\mathcal{T} _{w}\left( x\right) =\left\vert \sum\limits_{k\in 
\mathbb{Z}
}L\left( wx-t_{k}\right) -1\right\vert =\mathcal{O}\left( w^{-\theta _{0}}\right),%
\text{ as }w\rightarrow +\infty ,  \label{a6}
\end{equation}%
uniformly with respect to $x\in 
\mathbb{R}
,$ for some $\theta _{0}>0.$ Sometimes, a stronger condition
(instead of (\ref{a6})) is required, i.e., that%
\begin{equation}
\sum\limits_{k\in 
\mathbb{Z}
}L\left( u-t_{k}\right) =1,  \label{a7}
\end{equation}%
for every $u\in 
\mathbb{R}
.$ In this case, condition $\left( \chi 4\right) $ holds for every $\theta
_{0}>0.$ When $t_{k}=k$ (uniform sampling scheme) and $L$ is continuous, it is well known that (%
\ref{a7}) is equivalent to%
\begin{equation*}
\text{\ \ }\widehat{L}\left( 2 \pi k\right) :=
\left\{
\begin{array}{cc}
0, & \text{ \ \ }k\in 
\mathbb{Z}
\backslash \left\{ 0\right\} ,\\
1, & \text{\ }\ \ \ \ \ \ \ \ \text{\ }k=0,
\end{array}\right.
\end{equation*}%
where $\widehat{L}\left( \upsilon \right) :=\int\limits_{%
\mathbb{R}
}L\left( u\right) e^{-i\upsilon u}du,$ $\upsilon \in 
\mathbb{R}
,$ denotes the Fourier transform of $L$ (see \cite{Bardaro2, Butzer1, C-V5}).
\end{remark}

The rate of approximation for (linear) sampling Kantorovich operators in various
settings was studied in \cite{C-V5}. Also, a quantitative estimate for these
operators was obtained in \cite{C-V4} by using the modulus of continuity in
Orlicz spaces. 

  The general setting of Orlicz spaces allows us to directly
deduce the results concerning some quantitative estimates of approximation
in $L^{p}-$spaces (as in Corollary \ref{cor3.3}), together with some other useful spaces, as for examples Zygmund spaces and the exponential spaces, defined in Section \ref{sec2}.

In the case of approximation by linear sampling Kantorovich operators, considered in Remark \ref{Rk}, and below denoted by
$S_{w}^{\ast }$, we can immediately deduce, as a particular case, the following results.

\begin{corollary} \label{cor-linear}
Suppose that $M_p(L)<+\infty$, for $1 \leq p<\infty$.
Then, for every $f\in L^{p}\left( 
\mathbb{R}
\right) ,$ there holds%
\begin{eqnarray*}
\left \Vert S_{w}^{\ast }f-f\right \Vert _{p} &\leq &\delta ^{-1/p}[2m_{0,\Pi }\left(
	L\right)]^{\left( p-1\right) /p}\left[
\left \Vert L \right \Vert _{1}+M_{p}\left( L \right) \right]
^{1/p}\omega _{p}\left( f,1/w\right) \\
&&+\delta ^{-1/p}m_{0,\Pi }\left( L \right) \Delta ^{1/p}\omega
_{p}\left( f,\Delta /w\right) ,
\end{eqnarray*}%
for sufficiently large $w>0,$ where $m_{0,\Pi }\left( L \right)
<+\infty$.
Moreover, if $f\in Lip\left( \alpha ,p\right) ,$ with $0<\alpha \leq 1$, we have%
\begin{eqnarray*}
\left\Vert S_{w}^{\ast }f-f\right\Vert _{p} &\leq &\delta ^{-1/p}[2m_{0,\Pi }\left(
	L\right)]^{\left( p-1\right) /p}\left[
\left \Vert L \right \Vert _{1}+M_{p}\left( L \right) \right]
^{1/p} C_{1} \frac{1}{w^{\alpha }} \\
&&+\delta ^{-1/p}m_{0,\Pi }\left( L \right) C_{1}\Delta ^{1/p}\left( 
\frac{\Delta }{w}\right) ^{\alpha },
\end{eqnarray*}%
for sufficiently large $w>0,$ 
where $C_{1}>0$ is
the constant arising from the fact that $f$ belongs to $Lip(\alpha, p)$.
\end{corollary}
Note that, the estimates established in Corollary \ref{cor-linear} are sharper than those achieved in the general case of Theorem \ref{cor-Lip-varphi}.


\section{Examples of kernels} \label{sec4}

In this section, we give some concrete examples of the above nonlinear
sampling Kantorovich operators describing a natural procedure to
construct kernels. We will consider kernel functions of the form%
\begin{equation*}
\chi \left( wx-t_{k},u\right) =L\left( wx-t_{k}\right) g_{w}\left( u\right) ,
\end{equation*}%
where $\left( g_{w}\right) _{w>0}$ is a family of functions $g_{w}:%
\mathbb{R}
\rightarrow 
\mathbb{R}
$ satisfying $g_{w}\left( u\right) \rightarrow u$ uniformly on $%
\mathbb{R}
$ as $w\rightarrow +\infty $ and such that there exists a $\varphi -$%
function $\psi $ with%
\begin{equation} 
\left\vert g_{w}\left( u\right) -g_{w}\left( \upsilon \right) \right\vert
\leq \psi \left( \left\vert u-\upsilon \right\vert \right) ,  \label{n1}
\end{equation}%
for every $u,$ $\upsilon \in 
\mathbb{R}
$ and $w>0.$ Hence, assumptions $\left( \chi i\right) ,$ $i=1,..,4$ and $%
\left( Lj\right) ,$ $j=1,2$ can be summarized as follows.

$\left( \mathcal{L} 1\right) $ $k\rightarrow L\left( wx-t_{k}\right) \in
\ell ^{1}\left( 
\mathbb{Z}
\right) $, for every $x\in 
\mathbb{R}
$ and $w>0,$ $L$ is locally bounded in a neighborhood of the origin and
there exists $\beta _{0}>0$ such that%
\begin{equation}
m_{\beta _{0},\Pi }\left( L\right) :=\sup_{x\in 
\mathbb{R}
}\sum \limits_{k\in 
\mathbb{Z}
}L\left( wx-t_{k}\right) \left \vert wx-t_{k}\right \vert ^{\beta
_{0}}<+\infty ;  \label{n2}
\end{equation}

$\left( \mathcal{L} 2\right) $ $g_{w}\left( 0\right) =0,$ for every $w>0;$

$\left( \mathcal{L} 3\right) $ there exists $\theta _{0}>0$ such that%
\begin{equation*}
\mathcal{T} _{w}\left( x\right) :=\sup_{u\neq 0}\left \vert \frac{g_{w}\left(
u\right) }{u}\sum \limits_{k\in 
\mathbb{Z}
}L\left( wx-t_{k}\right) -1\right \vert =\mathcal{O}\left( w^{-\theta _{0}}\right),
\end{equation*}%
as $w\rightarrow +\infty ,$ uniformly with respect to $x\in 
\mathbb{R}
.$

Firstly, we show an example of sequence $\left( g_{w}\right) _{w>0}$
satisfying all the assumptions of the above theory.

\begin{example} \rm
Let us define $g_{w}\left( u\right) =u^{1-1/w}$ for every $u\in \left(
a,1\right) ,$ with $0<a<1/e$ and $g_{w}\left( u\right) =u$ otherwise, for $%
w>0.$ It is easily seen that $g_{w}\left( u\right) \rightarrow u$ uniformly
on $%
\mathbb{R}
,$ as $w\rightarrow +\infty .$ Note that if the function $L$ satisfies
condition (\ref{a7}), assumption $\left( \mathcal{L} 3\right) $ holds for $%
\theta _{0}=1.$ In fact, the function $g_{w}\left( u\right)
-u $ on $\left( a,1\right) $ achieves the maximum at $u_{0}:=\left( \frac{w-1%
}{w}\right) ^{w}$ for sufficiently large $w>0$ ($g_{w}\left( u\right) -u=0$
otherwise), then we have for every $u\in 
\mathbb{R}
$%
\begin{equation*}
\left\vert g_{w}\left( u\right) -u\right\vert \leq \left\vert g_{w}\left(
u_{0}\right) -u_{0}\right\vert =\left( \frac{w-1}{w}\right) ^{w}\left( \frac{%
1}{w-1}\right) \leq \frac{C}{w-1},
\end{equation*}%
for sufficiently large $w>0$ and for a suitable positive constant $C.$ Then%
\begin{equation*}
\sup_{u\neq 0}\left\vert \frac{g_{w}\left( u\right) }{u}-1\right\vert
=\sup_{u\in \left( a,1\right) }\frac{1}{\left\vert u\right\vert }\left\vert
g_{w}\left( u\right) -u\right\vert \leq a^{-1}\frac{C}{w-1}=\mathcal{O}\left( w^{-1}\right),
\end{equation*}%
as $w\rightarrow +\infty .$ Moreover, $g_w(u)$ satisfies (\ref{n1}) for sufficiently large $w>0$ and $\psi$ concave.
\end{example}

In addition, if we consider the particular case $g_{w}\left( u\right) =u,$ $%
u\in 
\mathbb{R}
,$ $w>0,$ the function $\psi $ corresponding to $\chi \left( x,u\right)
=L\left( x\right) u$ is $\psi \left( u\right) =u,$ $u\in 
\mathbb{R}
_{0}^{+}.$ In this case, our operators reduce to linear ones studied in
details in \cite{Bardaro2, C-V5, ANCOVI1} and as stated in Remark \ref{Rk},\
condition $\left( \mathcal{L} 3\right) $ becomes%
\begin{equation*}
\mathcal{T} _{w}\left( x\right) =\left\vert \sum\limits_{k\in 
\mathbb{Z}
}L\left( wx-k\right) -1\right\vert =\mathcal{O}\left( w^{-\theta _{0}}\right), \text{
as }w\rightarrow +\infty ,
\end{equation*}%
uniformly with respect to $x\in 
\mathbb{R}
,$ which is fulfilled for every $\theta _{0}>0$, if $\sum_{k \in \mathbb{Z}}L(u-k)=1$, for every $u \in \mathbb{R}$.

In order to construct a first example of function $L,$ we consider the
well-known Fej\'{e}r kernel, of the form 
\begin{equation*}
F\left( x\right) :=\frac{1}{2}\text{sinc}^{2}\left( \frac{x}{2}\right) ,%
\text{ }x\in 
\mathbb{R}
,
\end{equation*}%
where the sinc-function is the following%
\begin{equation*}
\text{sinc}(x) :=
\left\{
\begin{array}{cc}
\dfrac{\sin \left( \pi x\right) }{\pi x}, & \text{ \ \ }x\in 
\mathbb{R}
\backslash \left\{ 0\right\} ,\\
1, & \text{\ \ \ \ \ \ }x=0.
\end{array}\right.
\end{equation*}%

It can be easily seen that the function $F$ is non-negative and bounded, \ belongs to $%
L^{1}\left( 
\mathbb{R}
\right) $, with $\|F\|_1=m_{0,\Pi}(F) =1$ and satisfies the moment condition in (\ref{n2}) for every $%
0<\beta _{0}< 1$ (see, e.g., \cite{C-V2, C-V5, CACOVI3})$.$ In addition, it is easy to see that $M_\nu(F)<+\infty$, for any $0<\nu <1$, hence, as stated before, 
condition (\ref{a1}) is satisfied with $\alpha_0=(1-\alpha)\, \nu$, $0<\alpha<1$, and $M_1=M_\nu(F)$.
Furthermore, the Fourier transform of $F$ is given by%

\begin{equation*}
\widehat{F}\left( \upsilon \right) :=
\left\{
\begin{array}{cc}
1-\left\vert\upsilon /\pi \right\vert, & \text{ \ \ }\left\vert \upsilon \right\vert \leq\pi ,\\
0, & \text{\ \ \ \  }\left\vert \upsilon
\right\vert >\pi ,
\end{array}\right.
\end{equation*}%
(see \cite{Butzer1}). Then, by Remark \ref{Rk}, we have $\sum\limits_{k\in 
\mathbb{Z}
}F\left( u-k\right) =1$ for every $u\in 
\mathbb{R}
.$ Consequently, condition $\left( \mathcal{L} 3\right) $ reduces to $%
\sup\limits_{u\neq 0}\left\vert g_{w}\left( u\right) /u-1\right\vert
= \mathcal{O}\left( w^{-\theta _{0}}\right),$ as $w\rightarrow +\infty ,$ for some $%
\theta _{0}>0$, which is satisfied since $g_w(u)-u$ converges uniformly to zero, as $w\to+\infty$.

Then, considering e.g. a uniform sampling scheme, i.e., $t_{k}=k,$ $k\in 
\mathbb{Z}
,$ the corresponding nonlinear sampling Kantorovich operators based on Fej%
\'{e}r kernel are%
\begin{equation*}
\left( S_{w}^{F}f\right) \left( x\right) =\frac{1}{2}\sum\limits_{k\in 
\mathbb{Z}
}\text{sinc}^{2}\left( \frac{wx-k}{2}\right) g_{w}\left(
w\int\limits_{k/w}^{\left( k+1\right) /w}f\left( u\right) du\right) ,\text{
\ \ \ \ \ \ \ }x\in 
\mathbb{R}
,
\end{equation*}%
for every $w>0,$ where $f:%
\mathbb{R}
\rightarrow 
\mathbb{R}
$ is a locally integrable function such that the above series is convergent
for every $x\in 
\mathbb{R}
.$ For $S_{w}^{F}f,$ from Theorem \ref{Th1}, we obtain the following.

\begin{corollary}
Let $\varphi $ be a convex $\varphi -$function$.$ Suppose that $\varphi $
satisfies condition (H) with $\eta $ convex and $f\in L^{\varphi +\eta }$%
\bigskip $\left( 
\mathbb{R}
\right) $. Then, for every $0<\alpha <1$, there exist constant $\mu >0$, $\lambda _{0}>0$, such that 
$$
\hskip-8.5cm I^{\varphi }\left[ \mu \left( S_{w}^{F}f-f\right) \right] \leq 
$$
$$
\frac{%
1}{3}\left\{\omega \left(
f,1/w^{\alpha }\right) _{\eta }+ \widetilde{K}I^{\eta }\left[ \lambda
_{0}f\right]w^{-\alpha_0} +\omega \left( f,1/w\right) _{\eta }+ I^{\varphi }\left[
\lambda _{0}f\right]w^{-\theta _{0}} \right\},
$$
for sufficiently large $w>0,$ where $F$ is the Fej\'{e}r kernel, $%
\widetilde{K}>0$ is a suitable constant, and $\theta _{0}>0$ is the parameter of assumption $(\mathcal{L} 3)$.
\end{corollary}
Analogously, we may obtain a similar version of Corollary \ref{cor-Lip-varphi} for the operators $S^F_w f$.

The Fej\'{e}r kernel has unbounded support. Thus, to reconstruct a given
signal of $f$ by means of $S_{w}^{F}f,$ we need to compute an infinite
number of mean values $w\int\limits_{k/w}^{\left( k+1\right) /w}f\left(
u\right) du$ in order to evaluate the above operators at any fixed point $%
x\in 
\mathbb{R}
.$ Therefore, for a practical application of the above sampling series with $%
L$ having unbounded support, the sampling series must be truncated and this
leads to truncation errors which worsen the quality of the reconstruction.

However, considering kernels with $L$ having compact support, the truncation
error can be avoided. In this case, the infinite sampling series computed at
any fixed $x\in 
\mathbb{R}
$ reduces to a finite one. Important examples of kernels with compact
support can be generated by using the well-known B-spline of order $n\in 
\mathbb{N}
,$ given by%
\begin{equation*}
M_{n}\left( x\right) :=\frac{1}{\left( n-1\right) !}\sum\limits_{j=0}^{n}%
\left( -1\right) ^{j}\binom{n}{j}\left( \frac{n}{2}+x-j\right) _{+}^{n-1},
\end{equation*}%
where the function $\left( x\right) _{+}:=\max \left\{ x,0\right\} $ is the
positive part of $x\in 
\mathbb{R}
$ (see \cite{Butzer1, Bardaro2, Vinti, COVI17, COVI20, COVI21}).

The Fourier transform of $M_{n}$ is given by%
\begin{equation*}
\text{\ \ }\widehat{M_{n}}\left( \upsilon \right) := \text{sinc}^{n}\left( 
	\frac{\upsilon }{2\pi }\right),\text{ }\upsilon \in 
\mathbb{R}
.
\end{equation*}%
Then, we have $\sum\limits_{k\in 
\mathbb{Z}
}M_{n}\left( u-k\right) =1,$ for every $u\in 
\mathbb{R}
,$ by Remark \ref{Rk}. Therefore, condition $\left( \mathcal{L} 3\right) $
reduces to $\sup\limits_{u\neq 0}\left\vert g_{w}\left( u\right)
/u-1\right\vert =O\left( w^{-\theta _{0}}\right) ,$ as $w\rightarrow +\infty
,$ for some $\theta _{0}>0$, which is again satisfied. Obviously, each $M_{n}$ is bounded on $%
\mathbb{R}
,$ with compact support on $\left[ -n/2,n/2\right] ,$ and hence $M_{n}\in
L^{1}\left( 
\mathbb{R}
\right) ,$ for all $n\in 
\mathbb{N}
$, with $\|M_n\|_1=m_{0,\Pi}(M_n)=1$. Further, condition $\left( \mathcal{L} 1\right) $ is satisfied for every $%
\beta _{0}>0$ (see \cite{Bardaro2}).

In this case, the nonlinear sampling Kantorovich operators based upon the
B-spline kernel of order $n,$ with $t_{k}=k,$ $k\in 
\mathbb{Z}
,$ are given by%
\begin{equation*}
\left( S_{w}^{M_{n}}f\right) \left( x\right) =\sum\limits_{k\in 
\mathbb{Z}
}M_{n}\left( wx-k\right) g_{w}\left( w\int\limits_{k/w}^{\left( k+1\right)
/w}f\left( u\right) du\right) ,\text{ \ \ \ \ \ \ \ }x\in 
\mathbb{R}
,
\end{equation*}%
for every $w>0,$ where $f:%
\mathbb{R}
\rightarrow 
\mathbb{R}
$ is a locally integrable function such that the above series is convergent
for every $x\in 
\mathbb{R}
.$ From Corollary \ref{Cr}, we obtain the following.

\begin{corollary}
Let $\varphi $ be a convex $\varphi -$function satisfying condition (H) with 
$\eta $ convex and $f\in L^{\varphi +\eta }$\bigskip $\left( 
\mathbb{R}
\right) .$ Then, for every $0<\alpha <1,$ there exist
constants $\mu >0$, and $\lambda _{0}>0$, such that%
\begin{equation*}
I^{\varphi }\left[ \mu \left( S_{w}^{M_{n}}f-f\right) \right] \leq \frac{1}{3} \omega
\left( f,1/w^{\alpha }\right) _{\eta }+\frac{1}{3}\omega \left( f,1/w\right)
_{\eta }+\frac{I^{\varphi }\left[ \lambda _{0}f\right] }{3}w^{-\theta _{0}},
\end{equation*}%
for sufficiently large $w>0,$ where $\theta _{0}>0$ is the constant
of condition $(\mathcal{L} 3).$
\end{corollary}
As before, also for the operators $S_w^{M_n} f$ we may obtain a similar version of Corollary \ref{cor-Lip-varphi}. For other useful examples of kernels, see e.g., \cite{COGA1,COGA2,COGA3}.


\section*{Acknowledgments}

{\small The second and the third author are members of the Gruppo Nazionale per l'Analisi Matematica,
la Probabilit\`a e le loro Applicazioni (GNAMPA) of
the Istituto Nazionale di Alta Matematica (INdAM), and of the network RITA (Research ITalian network on Approximation). The second author has been partially supported within the 2020 GNAMPA-INdAM
Project ``Analisi reale, teoria della misura ed approssimazione per la ricostruzione di
immagini'', while the third author within
the projects: (1) Ricerca di Base 2018 dell'Universit\`a degli Studi
di Perugia - "Metodi di Teoria dell'Approssimazione,
Analisi Reale, Analisi Nonlineare e loro Applicazioni" ,
(2) Ricerca di Base 2019 dell'Universit\`a degli Studi di Perugia - "Integrazione, Approssimazione, Analisi Nonlineare e loro Applicazioni", 
(3) "Metodi e processi innovativi per lo sviluppo di una banca di
immagini mediche per fini diagnostici" funded by the
Fondazione Cassa di Risparmio di Perugia (FCRP), 2018, (4) "Metodiche di Imaging non invasivo mediante angiografia OCT sequenziale per lo studio delle Retinopatie degenerative dell'Anziano (M.I.R.A.)", funded by FCRP, 2019.
}

\end{document}